\documentclass{amsart}

\usepackage{amssymb}
\usepackage{amsmath}
\usepackage{latexsym} 
\usepackage{amsfonts}
\usepackage{amsthm}
\usepackage{mathrsfs}
\usepackage{verbatim} 
\usepackage[all,2cell,ps]{xy}
\usepackage{tikz}
\usetikzlibrary{shapes.geometric}
\tikzset{
  box/.style={
    regular polygon,
    regular polygon sides=6,
    minimum size=10mm,
    inner sep=0mm,
    outer sep=0mm,
    rotate=0,
  draw
  }
}
\usepackage[version=4]{mhchem}
\usepackage[pdftex,colorlinks=true, pdfstartview=FitV, linkcolor=magenta, citecolor=blue, urlcolor=blue]{hyperref}

\newcommand{\vl}{\;\vert\;}

\makeatletter
\def\adots{\mathinner{\mkern1mu\raise\p@
\vbox{\kern7\p@\hbox{.}}\mkern2mu
\raise4\p@\hbox{.}\mkern2mu\raise7\p@\hbox{.}\mkern1mu}}
\makeatother

\newcommand{\mz}{multiplicative Zagreb}

\newcommand{\E}[1]{E(#1)} 
\newcommand{\V}[1]{V(#1)} 
\newcommand{\dg}[1]{d(#1)} 
\newcommand{\z}[1]{\mathcal M_1(#1)} 
\newcommand{\zz}[1]{\mathcal M_2(#1)} 
\newcommand{\wang}[2]{\mathcal W_1^{#1}(#2)} 
\newcommand{\nz}[1]{\mathcal M_1^*(#1)} 
\newcommand{\hz}[1]{\mathcal{H}_1(#1)} 
\newcommand{\hzz}[1]{\mathcal{H}_2(#1)} 
\newcommand{\gz}[2]{\mathcal{M}_1^{#1}(#2)} 
\newcommand{\gzz}[2]{\mathcal{M}_2^{#1}(#2)} 
\newcommand{\NK}[1]{\mathcal NK(#1)} 

\newcommand{\pah}{Polycyclic Aromatic Hydrocarbons}
\newcommand{\PAH}{\ce{PAH_n}}

\newcommand{\bz}{Benzenoid}
\newcommand{\B}{$B_{m,n}$}

\newcommand{\defn}[1]{{\it #1}}

\newtheorem{theorem}{Theorem}[section]
\newtheorem{thm}[theorem]{Theorem}

\newtheorem{cor}[theorem]{Corollary}

\newtheorem{lemma}[theorem]{Lemma}

\theoremstyle{definition}

\title[Generalized Multiplicative Indices]
	{Generalized Multiplicative Indices of Polycyclic Aromatic Hydrocarbons and Benzeniod Systems}

\author[Kulli]{V.R. Kulli}
\address{V.R. Kulli, Department of Mathematics, Gulbarga University, Gullbarga 585106, India}
\email{vrkulli@gmail.com}

\author[Stone]{Branden Stone}
\address{Branden Stone, Department of Mathematics and Computer Science, Adelphi University, 1 South Avenue, Garden City, NY 11530-0701}
\email{bstone@adelphi.edu}

\author[Wang]{Shaohui Wang}
\address{Shaohui Wang, Department of Mathematics and Computer Science, Adelphi University, 1 South Avenue, Garden City, NY 11530-0701}
\email{shaohuiwang@yahoo.com}

\author[Wei]{Bing Wei}
\address{Bing Wei,  Department of Mathematics, University of Mississippi, University, MS 38655}
\email{bwei@olemiss.edu}


\begin{document}

\begin{abstract}
Many types of topological indices such as degree-based topological indices, distance-based topological indices and counting related topological indices are explored during past recent years.  Among degree based topological indices, Zagreb indices are the oldest one and studied well. In the paper, we define a generalized multiplicative version of these indices and compute exact formulas for Polycyclic Aromatic Hydrocarbons and Jagged-Rectangle Benzenoid Systems. 
\end{abstract}

\maketitle

\section{Introduction}

A molecular graph is a finite simple graph, representing the carbon-atom skeleton of an organic molecule of a hydrocarbon. The vertices of a molecular graph represent the carbon atoms and its undirected edges the carbon-carbon bounds. Throughout this paper $G = (V, E)$ is a connected molecular graph with vertex set $V = \V G$ and edge set $E = \E G$.   The degree $\dg v$ of a vertex $v$ is the number of vertices adjacent to $v$.\footnote{For more on this notation and terminology, the readers are referred to \cite{bondy}.}

Studying molecular graphs is a constant focus in chemical graph theory; an effort to better understand molecular structure. For instance, in 1947 H. Wiener \cite{wiener} introduced a topological index of a molecule, now called the \defn{Wiener index}. This index was originally defined as the sum of path distances between any two carbons in a saturated acyclic hydrocarbon. Since it's inception this index has been generalized to a variety of structures as well as used in Quantitative structure-activity relationship (QSAR) regression models \cite{fowler,matamala,randic,yang}. 

Some indices related to Wiener's work are the \defn{first and second multiplicative Zagreb indices} \cite{gutman}, respectively
\[
	\z G = \prod_{u \in \V G} \dg u^2 \text{ and } \zz G = \prod_{uv \in \E G} \dg u \dg v,
\]
and the \defn{Narumi-Katayama index} \cite{NK}
\[
	\NK G = \prod_{v \in \V G} \dg v.
\]
Like the Wiener index, these types of indices are the focus of considerable research in computational chemistry \cite{4,10, two,eight}. 
For example, in 2011 I.~Gutman \cite{4} characterized the multiplicative Zagreb indices for trees and determined the unique trees that obtained maximum and minimum values for $\z G$ and $\zz G$, respectively. S. Wang and the last author \cite{eight} then extended Gutman's result to the following index for $k$-trees,
\[
	\wang s G = \prod_{u \in \V G} \dg u^s. 
\]
Notice that $s=1,2$ is the Narumi-Katayama and Zagreb index, respectively. 

Based on the successful consideration of multiplicative Zagreb indices, M. Eliasi et al \cite{three} continued to define a new multiplicative version of the first Zagreb index as
\[
	\nz G = \prod_{uv \in \E G} \left( \dg u + \dg v \right).
\]
Furthering the concept of indexing with the edge set, the first author introduced the \defn{first and second hyper-Zagreb indices} of a graph \cite{four}. They are defined as
\[
	\hz G = \prod_{uv \in \E G} \left( \dg u + \dg v \right)^2 \text{ and } \hzz G = \prod_{uv \in \E G} \left( \dg u \dg v \right)^2.
\]
In this paper, we continue this generalization and define the \defn{general first and second \mz\ indices} of a graph $G$ as
\[
	\gz a G = \prod_{uv \in \E G} \left( \dg u + \dg v \right)^a \text{ and } \gzz a G = \prod_{uv \in \E G} \left( \dg u \dg v \right)^a.
\]
In Section \ref{pah} we determine the \mz\ and the general \mz\ indices for \pah\ (\PAH). Section \ref{bs} contains similar results for a jagged-rectangle \bz\ system (\B). 



\section{Results for Polycyclic Aromatic Hydrocarbons}\label{pah}

In this section, we focus on the molecular graph structure of the family of \pah, denoted \PAH.
These graphs of hydrocarbon molecules are defined recursively as follows. The 6-cycle with leaves at each vertex is \ce{PAH1} (\ce{C6H6}, benzene). The next element in the family, \ce{PAH2}, is given by deleting the leaves of \ce{PAH1} and gluing 6-cycles to each exterior edge, then adding leaves to each exterior vertex. We give the first three members of the family \PAH\ in Figure \ref{fig:pah}.

\begin{figure}
\begin{tabular}{ccc}
  
\begin{tikzpicture}[x=7.5mm,y=4.34mm]

\foreach \a/\b in {0/0}{
    
    \foreach \t in {71.5, 108.5, 251.5, 288.5} {
        \def\c{1.05};
        \def\d{2};
        \draw[color=black] 
            ({\a+\c*cos(\t)},{\b+\c*sin(\t)}) -- ({\a+\c*cos(\t)*\d},{\b+\c*sin(\t)*\d});
    }

    \foreach \t in {0,180} {
        \def\c{.65};
        \def\d{2};
        \draw[color=black] 
            ({\a+\c*cos(\t)},{\b+\c*sin(\t)}) -- ({\a+\c*cos(\t)*\d},{\b+\c*sin(\t)*\d});
        }

    \node[box,color=red,thick] at (\a,\b) {}; 

    }
\end{tikzpicture}

&

\begin{tikzpicture}[x=7.5mm,y=4.34mm]

\foreach \a/\b in {0/2,0/4,1/1,1/3,1/5,2/2,2/4} {
    
    \node[box] at (\a,\b) {}; 
    
    \foreach \t in {71.5, 108.5, 251.5, 288.5} {
        \def\c{1.05};
        \def\d{2};
        \draw[color=black] 
            ({\a+\c*cos(\t)},{\b+\c*sin(\t)}) -- ({\a+\c*cos(\t)*\d},{\b+\c*sin(\t)*\d});
    }

    \foreach \t in {0,180} {
        \def\c{.65};
        \def\d{2};
        \draw[color=black] 
            ({\a+\c*cos(\t)},{\b+\c*sin(\t)}) -- ({\a+\c*cos(\t)*\d},{\b+\c*sin(\t)*\d});
        }

    }

\node[box,color=red,thick] at (1,3) {};

\end{tikzpicture}

&

\begin{tikzpicture}[x=7.5mm,y=4.34mm]

\foreach \a/\b in {0/2,0/4,0/6,1/1,1/3,1/5,1/7,2/0,2/2,2/4,2/6,2/8,3/1,3/3,3/5,3/7,4/2,4/4,4/6} {

    \node[box] at (\a,\b) {}; 
    
    \foreach \t in {71.5, 108.5, 251.5, 288.5} {
        \def\c{1.05};
        \def\d{2};
        \draw[color=black] 
            ({\a+\c*cos(\t)},{\b+\c*sin(\t)}) -- ({\a+\c*cos(\t)*\d},{\b+\c*sin(\t)*\d});
    }

    \foreach \t in {0,180} {
        \def\c{.65};
        \def\d{2};
        \draw[color=black] 
            ({\a+\c*cos(\t)},{\b+\c*sin(\t)}) -- ({\a+\c*cos(\t)*\d},{\b+\c*sin(\t)*\d});
        }
    }

\foreach \a/\b in {1/3,1/5,2/2,2/4,2/6,3/3,3/5}
    \node[box,color=red,thick] at (\a,\b) {}; 

\end{tikzpicture}\\

\ce{PAH1} & \ce{PAH2} & \ce{PAH3}

\end{tabular}

	\caption{The first three elements of \PAH.}\label{fig:pah}
\end{figure}
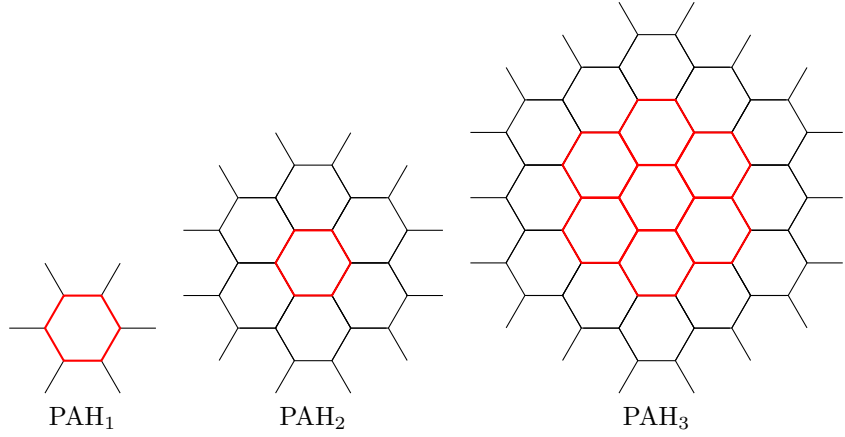

\begin{lemma}\label{lem:ref}
    Let $G = \PAH$\ be the molecular graph in the family of \pah. Then
    \begin{align*}
        |\V{G}| &= 6n^2+6n, \\
        |\E{G}| &= 9n^2+3n.
    \end{align*}
\end{lemma}

\begin{proof}
    We first need to show that $G$ has $6n$ leaves; we do this by induction on $n$. The result is clear for $n = 1,2$ and we assume \ce{PAH_{n-1}} has $6(n-1)$ leaves. To construct \PAH\ from \ce{PAH_{n-1}}, we attached $n-1$ hexagons; one between each pair of neighboring leaves. As six of these hexagons contribute two leaves to \PAH\ (and the rest contribute one), we see that \PAH\ has $6+6(n-1) = 6n$ leaves. 

    We are now able to find $|\V G|$ and $|\E G|$. Notice that each leaf in $G$ contributes 2 vertices. Removing these vertices from \PAH\ yields \ce{PAH_{n-1}}. Hence, by induction, 
    \[
        |\V{G}| = 2(6n)+(6(n-1)^2+6(n-1)) = 6n^2+6n. 
    \]
    Similarly, the leaves of \PAH\ contribute $3(6n)-6$ extra edges over \ce{PAH_{n-1}} (subtracting 6 accounts for the six hexagons contributing two leaves). Once again by induction,
    \[
        |\E{G}| = (3(6n)-6) + (9(n-1)^2 + 3(n-1)) = 9n^2+3n.
    \]
\end{proof}

We are now ready to compute the general indices of the molecular graph \PAH.

\begin{thm}\label{thm:gmz}
    Let $G = \PAH$\ be the molecular graph in the family of \pah. Then
	\begin{enumerate}
		\item $\gz a G = 4^{6an}\times 6^{(9n^2-3n)a}$;
		\item $\gzz a G = 3^{18an^2}$;
		\item $\wang s G = 3^{6sn^2}.$
	\end{enumerate}
\end{thm}
\begin{proof}
	According to Lemma \ref{lem:ref}, $G$ has $6n^2-6n$ vertices and $6n$ of those vertices are on leaves. With this information, we are able to partition of the vertex set of $G$ into two sets, one containing the vertices on leaves, and the other containing the rest of the vertices in the graph. We define them as
		\begin{align*}
			V_1 &= \{ v \in \V G \vl \dg v = 1 \}, \ |V_1| = 6n;\\	
			V_3 &= \{ v \in \V G \vl \dg v = 3 \}, \ |V_3| = 6n^2.
		\end{align*}
	Likewise, we obtain two partitions of the $9n^2+3n$ edges of $G$ as
		\begin{align*}
			E_1 &= \{ uv \in \E G \vl \dg u = 1, \dg v = 3\}, \ |E_1|=6n;\\
			E_3 &= \{ uv \in \E G \vl \dg u = \dg v = 3\}, \ |E_2|=9n^2-3n.
		\end{align*}
	Thus we are able to factor the products along these partitions. In particular, 
		\begin{align*}
			\gz a G &= \prod_{uv \in \E G} (\dg u + \dg v)^a \\
			      &= \prod_{uv \in E_1} (\dg u + \dg v)^a \times \prod_{uv\in E_3} (\dg u + \dg v)^a\\
			      &= [(1+3)^a]^{6n} \times [(3+3)^a]^{9n^2-3n}\\
			      &= 4^{6an} \times (6)^{(9n^2-3n)a}.		
		\end{align*}
	Similarly we have
		\begin{align*}
			\gzz a G &= \prod_{uv \in \E G} (\dg u\dg v)^a \\
			      &= \prod_{uv \in E_1} (\dg u\dg v)^a \times \prod_{uv\in E_3} (\dg u \dg v)^a\\
			      &= [(1\times 3)^a]^{6n} \times [(3\times 3)^a]^{9n^2-3n}\\
			      &= 3^{18an^2}.
		\end{align*}

	To see the last result we use the partitions of the vertex set to obtain
		\begin{align*}
			\wang s G &= \prod_{u \in \V G} \dg u^s  \\
			      &= \prod_{u \in V_1} \dg u^s \times \prod_{u\in V_3} \dg u^s\\
			      &= (1^s)^{6n} \times (3^s)^{6n^2}\\
			      &= 3^{6sn^2}.
		\end{align*}
\end{proof}

With this result we are able to calculate the remaining indices. 

\begin{cor}\label{cor:pah}
	Let $G = \PAH$\ be the molecular graph in the family of \pah. Then
	\begin{enumerate}
		\item $\z G = 3^{12n^2}$;
		\item $\zz G = 3^{18n^2}$;
		\item $\NK G = 3^{6n^2}$;
		\item $\nz G = 4^{6n}\times 6^{9n^2-3n}$;
		\item $\hz G = 4^{12n}\times 6^{18n^2-6n}$;
		\item $\hzz G = 3^{36n^2}$.
	\end{enumerate}
\end{cor}
\begin{proof}
	Each of the above indices are special cases of the general indices in Theorem~\ref{thm:gmz}. In particular we have, 
	\begin{align*}
		\z G &= \wang 2 G = 3^{12n^2}; \\
		\zz G &= \gzz 1 G = 3^{18n^2}; \\
		\NK G &= \wang 1 G = 3^{6n^2}; \\
		\nz G &= \gz 1 G = 4^{6n}\times 6^{9n^2-3n}; \\
		\hz G &= \gz 2 G = 4^{12n}\times 6^{18n^2-6n}; \\
		\hzz G &= \gzz 2 G = 3^{36n^2}.
	\end{align*}
\end{proof}

\section{Results for Benzenoid Systems}\label{bs}

We now focus on the molecular graph structure of a jagged-rectangle \bz\ system, denoted \B\ for all $m,n \in \mathbb N$. As can be seen in Figure~\ref{fig:bz} the rectangles \B\ are constructed be gluing $n+1$ chains of $m-1$ hexagons (or \ce{C6}) to $n$ chains of $m$ hexagons, alternating by starting with a $m-1$-chain of hexagons. This family of graphs was defined in \cite{nine}. In this section we will calculate the generalized multiplicative indices for these types of molecular graphs. 

\tikzset{
  box/.style={
    regular polygon,
    regular polygon sides=6,
    minimum size=8.7mm,
    inner sep=0mm,
    outer sep=0mm,
    rotate=90,
  draw
  }
}

\begin{figure}
\begin{tabular}{ccc}

\begin{tikzpicture}[x=7.5mm,y=4.34mm]

    \foreach \a/\b in {1/1,2/1,.5/2.5,1.5/2.5,2.5/2.5,1/4,2/4} {
        
        \node[box] at (\a,\b) {}; 

        }

    \node at (1,4) {\tiny 1};
    \node at (2,4) {\tiny 2};

    \node at (.5,2.5) {\tiny 1};        

\end{tikzpicture}

& 

\begin{tikzpicture}[x=7.5mm,y=4.34mm]

    \foreach \a/\b in {1/1,2/1,3/1,4/1,%
                        .5/2.5,1.5/2.5,2.5/2.5,3.5/2.5,4.5/2.5,%
                        1/4,2/4,3/4,4/4,%
                        .5/5.5,1.5/5.5,2.5/5.5,3.5/5.5,4.5/5.5,%
                        1/7,2/7,3/7,4/7,%
                        } {
        
        \node[box] at (\a,\b) {}; 
        }

    \node at (1,7) {\tiny 1};
    \node at (2,7) {\tiny 2};
    \node at (3,7) {\tiny 3};    
    \node at (4,7) {\tiny 4};

    \node at (.5,2.5) {\tiny 2};
    \node at (.5,5.5) {\tiny 1};        
\end{tikzpicture}

& 

\begin{tikzpicture}[x=7.5mm,y=4.34mm]

    \foreach \a/\b in {1/1,2/1,4/1,%
                        .5/2.5,1.5/2.5,2.5/2.5,4.5/2.5,%
                        .5/5.5,1.5/5.5,2.5/5.5,4.5/5.5,%
                        1/7,2/7,4/7,%
                        } {
        
        \node[box] at (\a,\b) {}; 

        }

    \node at (3,1) {$\cdots$};
    \node at (3.5,2.5) {$\cdots$};
    \node at (1,4.2) {$\vdots$};
    \node at (2,4.2) {$\vdots$};
    \node at (3.5,5.5) {$\cdots$};
    \node at (3,7) {$\cdots$};

    \node at (1,7) {\tiny 1};
    \node at (2,7) {\tiny 2};
    \node at (4,7) {\tiny $m-1$};

    \node at (.5,2.5) {\tiny $n$};
    \node at (.5,5.5) {\tiny 1};
\end{tikzpicture}\\

$B_{3,1}$ & $B_{5,2}$ & $B_{m,n}$

\end{tabular}
\caption{The molecular graphs of a jagged-rectangle \bz\ system.}\label{fig:bz}
\end{figure}
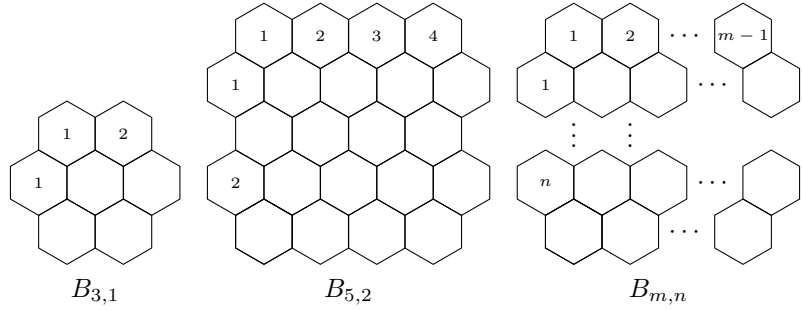

As to the general indices for this system, we have the following result.

\begin{thm}\label{thm:gbz}
	Let $G =$\B\ be a molecular graph of a jagged-rectangle \bz\ system. Then
	\begin{enumerate}
		\item $\gz a G = 4^{a(2n+4)}\times 5^{a(4m+4n-4)}\times 6^{a(6mn+m-5n-4)}$;
		\item $\gzz a G = 4^{a(2n+4)}\times 6^{a(4m+4n-4)}\times 9^{a(6mn+m-5n-4)}$;
		\item $\wang s G = 2^{(2m+4n+2)s} \times 3^{(4mn+2m-2n-4)s}$
	\end{enumerate}
\end{thm}
\begin{proof}
	We first calculate the number of vertices and edges of $G$. To do this notice that the number of vertices in the top row of $m-1$ hexagons (oriented according to Figure~\ref{fig:bz}) is $4m-2$. As there are $n+1$ of rows containing $m-1$ hexagons in the graph, we have counted $(4m-2)(n+1) = 4mn-2n+4m-2$ vertices so far. The only remaining vertices are on the left and right ends of the $n$ rows containing $m$ hexagons; there are $4n$ of these. Hence we get
	\[
		|\V G| = (4mn-2n+4m-2) + (4n) = 4mn+4m+2n-2.
	\]
	To find the number of edges, we partition $\V G$ into two sets, vertices of degree 2 and 3 respectively, 
	\begin{align*}
		V_2 &= \{ v \in \V G \vl \dg v = 2 \}, \ |V_2| = 2m+4n+2;\\
		V_3 &= \{ v \in \V G \vl \dg v = 3 \}, \ |V_3| = 4mn+2m-2n-4.
	\end{align*}
	As the total degree of the graph is equal to twice the number of edges, we know that 
	\[
		|\E G| = \frac 1 2 [2(2m+4n+2)+3(4mn+2m-2n-4)] = 6mn+5m+n-4.
	\]

	Similar to the proof of Theorem~\ref{thm:gmz}, we will calculate the indices by factoring along partitions of the vertex and edge sets. To see the last result we use the partitions of the vertex set to obtain
		\begin{align*}
			\wang s G &= \prod_{u \in \V G} \dg u^s  \\
			      &= \prod_{u \in V_2} \dg u^s \times \prod_{u\in V_3} \dg u^s\\
			      &= (2^s)^{2m+4n+2} \times (3^s)^{4mn+2m-2n-4}\\
			      &= 2^{(2m+4n+2)s} \times 3^{(4mn+2m-2n-4)s}.
		\end{align*}

	For the remaining results, we create three partitions of the edge set of the molecular graph $G$.
	\begin{align*}
		E_2 &= \{ uv \in \E G \vl \dg u = \dg v = 2\}, \ |E_2|=2n+4;\\
		E_{2,3} &= \{ uv \in \E G \vl \dg u = 3, \dg v = 2\}, \ |E_{2,3}|=4m+4n-4;\\
		E_3 &= \{ uv \in \E G \vl \dg u = \dg v = 3\}, \ |E_3|=6mn+m-5n-4.
	\end{align*}
	It is not hard to see that $|E_2|=2n+4$. To see the number of elements in $E_{2,3}$, notice that there are $4n+8$ vertices  of degree 2 with a unique adjacent vertex of degree 3. Further, there are $2m - 6$ vertices with two distinct adjacent vertices of degree 3. Hence $|E_{2,3}|= 4n+8 + 2(2m-6) = 4m+4n-4$. Subtracting these values from $|\E G|$ yields $|E_3|$. 

	Using this edge partition, we are able to calculate $\gz a G$ as 
	\begin{align*}
		\gz a G &= \prod_{uv \in \E G} (\dg u + \dg v)^a \\
		      &= \prod_{uv \in E_2} (\dg u + \dg v)^a \times \prod_{uv\in E_{2,3}} (\dg u + \dg v)^a \times \prod_{uv\in E_3} (\dg u + \dg v)^a\\
		      &= [(2+2)^a]^{2n+4} \times [(3+2)^a]^{4m+4n-4}\times [(3+3)^a]^{6mn+m-5n-4}\\
		      &= 4^{a(2n+4)}\times 5^{a(4m+4n-4)}\times 6^{a(6mn+m-5n-4)}.
	\end{align*}
	To see the second result, we have
	\begin{align*}
		\gzz a G &= \prod_{uv \in \E G} (\dg u\dg v)^a \\
		      &= \prod_{uv \in E_2} (\dg u\dg v)^a \times \prod_{uv\in E_{2,3}} (\dg u \dg v)^a\times \prod_{uv\in E_3} (\dg u \dg v)^a\\
		      &= 4^{a(2n+4)}\times 6^{a(4m+4n-4)}\times 9^{a(6mn+m-5n-4)}.
	\end{align*}
\end{proof}

As an immediate corollary all other indices in this paper are obtained. 

\begin{cor}\label{cor:bz}
	Let $G =$\B\ be a molecular graph of a jagged-rectangle \bz\ system. Then
	\begin{enumerate}
		\item $\z G = 4^{2m+4n+2}\times 9^{4mn+2m-2n-4}$;
		\item $\zz G = 4^{2n+4}\times 6^{4m+4n-4}\times 9^{6mn+m-5n-4}$;
		\item $\NK G = 2^{2m+4n+2}\times 3^{4mn+2m-2n-4}$;
		\item $\nz G = 4^{2n+4}\times 5^{4m+4n-4}\times 6^{6mn+m-5n-4}$;
		\item $\hz G = 4^{2(2n+4)}\times 5^{2(4m+4n-4)}\times 6^{2(6mn+m-5n-4)}$;
		\item $\hzz G = 4^{2(2n+4)}\times 6^{2(4m+4n-4)}\times 9^{2(6mn+m-5n-4)}$.
	\end{enumerate}
\end{cor}
\begin{proof}
	Each of the above indices are special cases of the general indices in Theorem~\ref{thm:gbz}. In particular we have, 
	\begin{align*}
		 \z G &= \wang 2 G = 4^{2m+4n+2}\times 9^{4mn+2m-2n-4}; \\
		\zz G &= \gzz 1 G = 4^{2n+4}\times 6^{4m+4n-4}\times 9^{6mn+m-5n-4};\\
		\NK G &= \wang 1 G = 2^{2m+4n+2}\times 3^{4mn+2m-2n-4};\\
		\nz G &= \gz 1 G = 4^{2n+4}\times 5^{4m+4n-4}\times 6^{6mn+m-5n-4};\\
		\hz G &= \gz 2 G = 4^{2(2n+4)}\times 5^{2(4m+4n-4)}\times 6^{2(6mn+m-5n-4)};\\
		\hzz G &= \gzz 2 G = 4^{2(2n+4)}\times 6^{2(4m+4n-4)}\times 9^{2(6mn+m-5n-4)}.		
	\end{align*}
\end{proof}



\end{document}